\documentclass[10pt]{amsart}
\usepackage{amssymb,latexsym,amsmath,amsthm,enumitem,mathrsfs,geometry}
\usepackage{fullpage}
\usepackage{fancyhdr}

\usepackage{mathrsfs}
\usepackage{hyperref}
\usepackage{lineno}
\usepackage{diagbox}
\usepackage{array}
\usepackage{tikz}
\usetikzlibrary{arrows}
\usetikzlibrary{positioning}
\usepackage{float}
\usepackage{pgfplots}

\newtheorem{thm}{Theorem}[section]
\newtheorem*{thm*}{Theorem}
\newtheorem{prop}[thm]{Proposition}
\newtheorem{lemma}[thm]{Lemma}

\newcommand{\beq}{\begin{equation}}
\newcommand{\eeq}{\end{equation}}

\usepackage[titletoc,title]{appendix}

\newcommand{\Q}{\mathbb{Q}}
\newcommand{\R}{\mathbb{R}}
\newcommand{\F}{\mathbb{F}}
\newcommand{\Z}{\mathbb{Z}}
\newcommand{\C}{\mathbb{C}}

\newcommand{\Sym}{\mathrm{Sym}}

\newcommand{\kommentar}[1]{}

\newtheorem*{remark*}{Remark}

\begin{document}
\numberwithin{equation}{section}

\title{Extremal primes for elliptic curves without complex multiplication}

\author[David]{C. David}
\address{Department of Mathematics, Concordia University, 1455 de Maisonneuve West, Montreal, QC H3G 1M8, Canada}
\email{chantal.david@concordia.ca}

\author[Gafni]{A. Gafni}
\address{Department of Mathematics, The University of Mississippi, Hume Hall 305, University, MS 38677, USA}
\email{ayla.gafni@gmail.com}

\author[Malik]{A. Malik}
\address{Department of Mathematics, Rutgers University, 110 Frelinghuysen Rd., Piscataway, NJ 08854, USA}
\email{amita.malik@rutgers.edu}

\author[Prabhu]{N. Prabhu}
\address{Department of Mathematics and Statistics, Queen's University, 48 University Avenue, Kingston, ON K7L 3N6, Canada}
\email{neha.prabhu@queensu.ca}

\author[Turnage-Butterbaugh]{C. Turnage-Butterbaugh}
\address{Department of Mathematics and Statistics, Carleton College, 1 North College Street, Northfield, MN 55057, USA}
\email{cturnageb@carleton.edu}

\thanks{}

\date{\today}

\begin{abstract} Fix an elliptic curve $E$ over $\mathbb{Q}$. An {\it extremal prime} for $E$ is a prime $p$ of good reduction such that the number of rational points on $E$ modulo $p$ is maximal or minimal in relation to the Hasse bound, i.e. $a_p(E) = \pm \left[ 2 \sqrt{p} \right]$. Assuming that all the symmetric power $L$-functions associated to $E$ have analytic continuation for all $s \in \C$, satisfy the expected functional equation and the Generalized Riemann Hypothesis, we provide upper bounds for the number of extremal primes when $E$ is a curve without complex multiplication. 
In order to obtain this bound, we use explicit equidistribution for the Sato-Tate measure as in the work of Rouse and Thorner \cite{RT2017}, and refine certain intermediate estimates taking advantage of the fact that extremal primes are less probable than primes where $a_p(E)$ is fixed because of the Sato-Tate distribution.
\end{abstract}

\maketitle

\section{Introduction}
Let $E$ denote an elliptic curve over $\Q$. For a prime $p$ of good reduction, $E$ reduces to an elliptic curve over the finite field $\F_p$, and we denote by $a_p(E)$ the trace of the Frobenius automorphism acting on the points of $E$ over ${\overline {\F}}_p$. Then  $a_p(E) = p+1-\# E(\F_p)$, and $|a_p(E) | \leq 2 \sqrt{p}$ (the Hasse bound). The following conjecture for the distribution of the normalized traces $a_p(E)/2 \sqrt{p}$ in $[-1,1]$ was formulated independently by Sato and Tate.

\begin{thm}[Sato-Tate conjecture]\label{ST} Let $E$ be an elliptic curve without complex multiplication over $\Q$. Let $\alpha, \beta \in \R$ with $0 \leq \alpha \leq \beta \leq 1$.
Then, as $x \rightarrow \infty$, 
$$\frac{1}{\pi(x)} \# \left\{ p \leq x \;:\; \frac{a_p(E)}{2 \sqrt{p}} \in ( \alpha, \beta ) \right\} \sim \frac{2}{\pi} \int_\alpha^\beta \sqrt{1-t^2} \,dt .$$
\end{thm}

If $E$ has at least one prime of multiplicative reduction, the Sato-Tate conjecture was proven by Taylor \cite{Taylor2008}, in collaboration with Clozel, Harris and Shepherd-Barron \cite{CHT2008,HSBT2010}.

We study in this paper a refinement of the Sato-Tate conjecture concerning the distribution of the primes $p$ which fall at the extremes of this distribution, i.e. the primes $p$ such that $a_p(E) = \pm [ 2 \sqrt{p} ]$, where for any real number $y$, $[ y ]$ denotes the integer part of $y$. Then $\# E(\F_p)$ is maximal when $a_p(E) = - [2 \sqrt{p}]$ and minimal when $a_p(E) =  [2 \sqrt{p}]$.

Extremal primes were first studied by James {\it et al.} \cite{James2016} who conjectured (as refined by James and Pollack \cite{JP2017}) that, as $x\to \infty$,
\begin{eqnarray}\nonumber
&& \hspace{-1cm} \# \left\{ p \leq x \;:\; a_p(E) = [ 2 \sqrt{p} ] \right\} \\
&\sim& \begin{cases}  \displaystyle \frac{8}{3\pi} \; \frac{x^{1/4}}{\log{x}}, & \mbox{if $E$ does not have complex multiplication,} \\  \\ \label{conj-EP-CM}
\displaystyle \frac{2}{3\pi} \; \frac{x^{3/4}}{\log{x}}, & \mbox{if $E$ has complex multiplication.} \end{cases}
\end{eqnarray}

By symmetry, an analogous conjecture has been stated for extremal primes with $a_p(E) = -[ 2 \sqrt{p} ] $.
It is enlightening to compare this conjecture with another refinement of the Sato-Tate conjecture, namely the Lang-Trotter conjecture.
For any fixed value $h \in \Z$, the Lang-Trotter conjecture \cite{LT1976} predicts that
\begin{eqnarray} \label{conj-EP}
\pi_{E, h}(x) = \# \left\{ p \leq x \;:\; a_p(E) = h  \right\} \sim C_{E,h} \frac{x^{1/2}}{\log{x}}
\end{eqnarray}
as $x\to \infty$, where $C_{E,h}$ is a specific constant \footnote{If $h=0$, then it is additionally assumed that $E$ does not have complex multiplication. The case  $h=0$ and $E$ with complex multiplication was treated in \cite{Deuring}.}. Comparing \eqref{conj-EP-CM} and \eqref{conj-EP}, we notice that for non-CM curves, there are expected to be fewer extremal primes than primes with a fixed value of $a_p(E)$, since the extremal primes are at the edge of the Sato-Tate distribution of Theorem \ref{ST}, where the measure is small.
On the other hand, for CM curves, an excess of extremal primes is predicted, since in this case, the measure for the distribution of $a_p(E)/2 \sqrt{p}$ in $[\alpha, \beta] \subseteq [-1,1] \setminus\{0\}$ is given by 
$$
\mu_{\text CM}([\alpha, \beta]) = \frac{1}{2 \pi} \int_\alpha^\beta \frac{dt}{\sqrt{1-t^2}}.
$$

The asymptotic \eqref{conj-EP-CM} for CM curves was proven by James and Pollack \cite{JP2017}. In a subsequent paper by Agwu {\it et al.}\cite{AHJKL} the authors obtained asymptotics for a refined question for CM curves, namely the primes where $a_p(E)$ falls within a small range of the end of the Hasse interval. In this article, we focus on the case of non-CM curves.

Like the Lang-Trotter conjecture, the asymptotic \eqref{conj-EP-CM} for non-CM curves seems to be out of reach with current techniques. An asymptotic was proven to hold on average for non-CM elliptic curves $E / \Q$ in the Ph.D. thesis of Giberson \cite{Giberson} (see also \cite{GJ}). However, no non-trivial upper bounds are known for a single curve $E/\Q$. The goal of this paper is to obtain such upper bounds. 

Let $N_E$ denote the conductor of the elliptic curve $E$, and define
\begin{equation} \label{normL}
L(s,E) = \prod_{p \nmid N_E}\left( 1-\frac{\alpha_p(E)}{p^s}\right)^{-1}\left( 1-\frac{\overline{\alpha}_p(E)}{p^s}\right)^{-1} \; \prod_{p \mid N_E} \left( 1-\frac{a_p(E)}{p^s}\right)^{-1}
\end{equation}
where we have normalized the $L$-function so that $\alpha_p(E),\,\overline{\alpha}_p(E)$ satisfy 
$$\#E(\mathbb{F}_p)=p+1-\sqrt{p}(\alpha_p(E)+\overline{\alpha}_p(E)) \;\; \mbox{for $p \nmid N_E$}.$$ For any integer $n \geq 0$, the symmetric power $L$-functions of $E$ are given by
\[
L(s,\Sym^n(E)) = \prod_{p|N_E}L_p(s,\Sym^n(E))\prod_{p\nmid N_E}\prod_{j=0}^{n}\left(1-\frac{\alpha_p(E)^j\overline{\alpha}_p(E)^{n-j}}{p^s} \right)^{-1},
\]
where the Euler factors $L_p(s,\Sym^n(E))$ at the bad primes are described in Appendix \ref{A3}. 

The aforementioned proof of the Sato-Tate conjecture was obtained by proving that if $E$ has at least one prime of multiplicative reduction, then the functions  $L(s,\Sym^n(E))$ have meromorphic continuation to the whole complex plane, satisfy the functional equation \ref{FE-Sym}, and are analytic and non-zero for $\mbox{Re}(s) \geq 1$.
(See \cite[Theorem B]{Taylor2008}, with the difference there that the $L$-functions are not normalized.) 
To get an effective version of the Sato-Tate conjecture in \cite{RT2017}, the authors need to assume each $L(s,\Sym^n(E))$ has analytic continuation to the whole complex plane, and satisfies the Generalized Riemann Hypothesis (GRH). For convenience, we  include the function $L(s,\Sym^0(E)) = \zeta(s)$, which is analytic except for a simple pole at $s=1$.

 Under the same hypotheses, one can also obtain upper bounds for the Lang-Trotter conjecture.
This was carried out by K. Murty \cite{VKMurty} and extended by Bucur and Kedlaya to arbitrary motives  \cite{BC}. These results were improved recently by Rouse and Thorner \cite{RT2017}, who proved (under the same hypotheses as in Theorem \ref{main}, stated below) that
$$   \pi_{E,h}(x) \ll_{E,h} x^{3/4} (\log{x})^{-1/2}.
 $$
   In our case, taking advantage of the fact that extremal primes fall at the edge of the Sato-Tate interval, we refine the work of Rouse and Thorner to obtain a better upper bound for the number of extremal primes.
\begin{thm} \label{main}
	Let $E$ be a non-CM elliptic curve over $\Q$. Assume that for any $n \geq 0$, the $L$-functions $L(s,\Sym^n(E))$ have analytic continuation to the entire complex plane (except for a simple pole at $s=1$ when $n=0$), satisfy the functional equation \eqref{FE-Sym} and the Generalized Riemann Hypothesis. 
	 Then \begin{equation*}
	\# \{  x < p \leq  2x \;:\; a_p(E) = [2\sqrt{p}] \} \ll_E {x^{1/2}}.
	\end{equation*}
\end{thm}
\vspace{.2in}

\noindent{\bf Acknowledgements} This work was initiated at the WIN4 workshop held at the Banff International Research Station in August 2017, and the authors thank the organizers of the workshop and BIRS. The authors also thank Dimitris Koukoulopoulos, Jesse Thorner, and A. Raghuram for helpful discussions related to this paper, and Igor Shparlinski for important comments on a previous version of this paper. Turnage-Butterbaugh thanks the Max Planck Institute for Mathematics for their hospitality and support during July 2018. 
The authors would also like to thank the anonymous referee for useful comments which have improved the exposition of the paper.

\section{Explicit equidistiribution for the Sato-Tate measure}

In this section, we prove upper bounds on Fourier coefficients of certain trigonometric polynomials which approximate characteristic functions of intervals. By considering short intervals at the edge of the Sato-Tate distribution, we obtain upper bounds that are stronger than what one obtains for general intervals. 

We first briefly review  classical results on explicit equidistribution. We refer the reader to \cite[Chapter 1]{Mont} for a detailed exposition of trigonometric polynomials approximating the characteristic function on subintervals of $[0,1]$ with respect to the uniform measure, and for the notations in this section. By the change of variable $t = \cos\theta$, we can view the Sato-Tate measure for non-CM curves, stated in Theorem \ref{ST}, as the measure on $[0,\pi]$ given by
$$ 
\mu_{ST}([\alpha, \beta]) = \frac{2}{\pi} \int_{\alpha}^{\beta} \sin^2\theta \;d \theta,
$$
for $[\alpha, \beta] \subseteq [0, \pi]$.
To approximate the characteristic function of intervals in $[0, \pi]$ with respect to the Sato-Tate measure, one uses the Chebyshev polynomials of the second kind, denoted by $U_n$ and defined by the recurrence relation
\begin{align*}
U_0(x)&=1\\
U_1(x)&=2x\\
U_{n}(x)&=2xU_{n-1}(x)-U_{n-2}(x).
\end{align*}
We remark that the polynomials form an orthonormal basis with respect to the Sato-Tate measure on $[0,\pi]$. We refer to \cite{RT2017} for the proof of the following lemma, which follows directly from explicit uniform equidistribution.
	\begin{lemma}[{\cite[Lemma 1.3]{RT2017}}] \label{lemma-RT}
		Let $I = [\alpha, \beta] \subseteq[0,\pi]$, and let $M$ be a positive integer. There exist trigonometric polynomials $$F^{\pm}_{I,M}(\theta) = \sum\limits_{n=0}^M \hat{F}^{\pm}_{I,M}(n)U_n(\cos\theta)$$ that satisfy the following properties.
		\begin{itemize}
			\item For $0\leq \theta \leq \pi$, we have 
			\begin{equation*}\label{approximating-char-function}
			F^{-}_{I,M}(\theta)\leq \chi_I(\theta) \leq F^{+}_{I,M}(\theta).
			\end{equation*} 
			\item We have 
			\begin{equation*}
						|\hat{F}^{\pm}_{I,M}(0) - \mu_{ST}(I)|\leq \frac{4}{M+1}.
			\end{equation*}
			\item For $1\leq n\leq M$, we have 
			\begin{equation*}
			|\hat{F}^{\pm}_{I,M}(n)| \leq 4\left(\frac{1}{M+1} + \min\left\{\frac{\beta-\alpha}{2\pi}, \frac{1}{\pi n}\right\}\right).
			\end{equation*}
		\end{itemize}
	\end{lemma}
	The above lemma is valid for any interval $[\alpha, \beta] \subseteq [0, \pi]$. In our case, however, we are interested in counting at the edge of the Sato-Tate interval, where the measure is very small. More precisely, we will consider intervals where $t=\cos\theta$ is close to $1$, i.e. where $\theta$ is close to $0$. In this way, we obtain the following sharper estimate for the Fourier coefficients $\hat{F}^+_{I,M}(n)$, which will be key inputs in the proof of Theorem \ref{main}.

	\begin{prop}\label{improved-FC}
		\label{new-Fourier} Assume the setting and notation of Lemma \ref{lemma-RT}.
		If $I=[0, \frac{1}{M}] \subseteq [0,\pi]$, then for $0 \leq n \leq M$, 
		$$\hat{F}^+_{I,M}(n) \ll \frac{1}{M^2}.$$
	\end{prop}

\begin{proof}
	
	Since the $U_n$ are an orthonormal basis, for $0 \leq n \leq M$ we have 
	\begin{align} \nonumber
	\hat{F}^+_{I, M}(n) &= \int_{0}^\pi F^+_{I, M} (\theta) U_n(\cos \theta) \sin^2{\theta} \; d\theta \\ \label{2-integrals}
	&= \int_{0}^\pi  \chi_{I} (\theta) U_n(\cos \theta) \sin^2{\theta} \; d\theta+ \int_{0}^\pi \left( F^+_{I, M} (\theta)  - \chi_I(\theta) \right) U_n(\cos \theta) \sin^2{\theta} \; d\theta ,
	\end{align}
	The first integral of \eqref{2-integrals} is easily bounded by using the fact that $$U_n(\cos{\theta}) = \frac{\sin{\left( (n+1) \theta \right)}}{\sin \theta},$$ yielding
	\begin{eqnarray*}
		\int_0^\pi \chi_{I}(\theta) \; U_n(\cos{\theta})  \sin^2{\theta} \; d\theta &=&  \int_0^{1/M} \sin{\left((n+1) \theta \right)}  \sin{\theta} \; d\theta 
		\ll \frac{1}{M^2}.
	\end{eqnarray*}

For the second integral of \eqref{2-integrals}, we must bound the distance between the approximation of length $M$ and $\chi_I(\theta)$. We recall the definition of $F^+_{I,M}(\theta)$. For any $J = [0, \beta] \subseteq [0,1]$, it is straightfoward to see that
\[
\chi_J(x) = \beta + s(x-\beta) + s( - x),
\]
where $s(x)$ denotes the saw-tooth function
\[
s(x) = \begin{cases} \left\{ x \right\} - 1/2 & \mbox{if $x \not\in \Z$} \\ 0 & \mbox{if $x \in \Z$}. \end{cases}
\]
Next, recall that the Beurling polynomial, $B_M(x)$, is defined by
\[
B_M(x)=V_M(x) + \frac{1}{2(M+1)} \Delta_{M+1}(x),
\]
where $\Delta_M(x)$ is the Fejer kernel given by 
\[
\Delta_{M}(x) = \frac{1}{M} \left( \frac{\sin{\pi M x}}{\sin{\pi x}} \right)^2,
\]
and $V_M(x)$ is the Vaaler polynomial given by 
\begin{align*}V_M(x)&= \dfrac{1}{M+1}\sum\limits_{k=1}^{M}\left(\dfrac{k}{M+1} - \dfrac{1}{2} \right) \Delta_{M+1}\left(x - \dfrac{k}{M+1} \right)\\
	&\hspace{.25in}+\dfrac{1}{2\pi (M+1)} \sin(2\pi (M+1)x) -\dfrac{1}{2\pi} \Delta_{M+1}(x)\sin 2\pi x.
\end{align*}
We set 
\[
		S^+_{J,M}(x) = \beta + B_M(x-\beta) + B_M(- x).
\]
When $I = [0, 1/M]$, setting the interval $J = [0,  1/(2 \pi M)] \subseteq [0,1/2]$, we have that
\[
F^+_{I,M}(\theta) = S^+_{J,M}  \left( \frac{\theta}{2 \pi} \right) + S^+_{J,M} \left( - \frac{\theta}{2 \pi} \right).
\]
With the change of variable $x = \theta/2 \pi$ and $\beta = 1/(2 \pi M)$, we find
	\begin{eqnarray*}
		F^+_{I,M}(\theta) - \chi_I(\theta) &=& S^+_{J, M}(x ) + S^+_{J,M}(-x ) -  \chi_J(x) - \chi_J(-x) \\
		&=& B_M(x-\beta) - s(x-\beta) + B_M(-x) - s(-x) \\
		&& + B_M(-x-\beta) - s(-x-\beta) + B_M(x) - s(x).
	\end{eqnarray*}
	The second integral of \eqref{2-integrals} then writes as
\begin{eqnarray*}
&&\hspace{-.5in} 2\pi \int_{-1/2}^{1/2} \left( B_M(x) - s(x) \right) \; U_n(\cos 2 \pi x) \; \sin^2{(2 \pi x) } \; dx; \\
&&+\; 2\pi \int_{-1/2}^{1/2}  \left( B_M(x-\beta) - s(x-\beta) \right) \; U_n(\cos 2 \pi x) \; \sin^2{(2 \pi x) } \; dx.\end{eqnarray*}
	We now compute the second integral in the above expression, since upon taking $\beta = 0$ we will recover the first integral. From the definition of the polynomial $B_M$, we have
	\begin{eqnarray}  \nonumber
	&&\hspace{-1in} \int_{-1/2}^{1/2} \left( B_M(x-\beta) - s(x-\beta) \right) \; U_n(\cos 2 \pi x) \; \sin^2{(2 \pi x) } \; dx \\
	&=& \label{vaaler}
	\int_{-1/2}^{1/2} \left( V_M(x-\beta) - s(x-\beta) \right) \; U_n(\cos 2 \pi x) \; \sin^2{(2 \pi x) } \; dx \\ \label{fejer}
	&& +  \; \frac{1}{2(M+1)} \int_{-1/2}^{1/2} \Delta_{M+1}(x-\beta) \; U_n(\cos 2 \pi x) \; \sin^2{(2 \pi x) } \; dx.
	\end{eqnarray}
	To estimate \eqref{vaaler}, we use the bound (for  $|x|\leq 1/2$)
	\begin{eqnarray*}
		V_M(x-\beta) - s(x-\beta) &\ll& \min{ \left( 1, \frac{1}{M^3 |x-\beta|^3} \right)}  \\[.15in]&\ll& 			\begin{cases} 
	1 & \text{if }|x-\beta| < 1/M \\[.15in]
	(M |x-\beta|)^{-3} & \text{if }1/M \leq |x-\beta| \leq 1/2, \end{cases}
	\end{eqnarray*}
	and find, for $\beta = 1/(2 \pi M) \ll 1/M$, that
	\begin{align*}
	\int_{-1/2}^{1/2} &\left( V_M(x-\beta) - s(x-\beta) \right) \; U_n(\cos 2 \pi x) \; \sin^2{(2 \pi x) } \; dx \\
	&\ll \int_{\beta-1/M}^{\beta +1/M} \left|  \sin{\left((n+1) 2 \pi x\right)} \;\sin{(2 \pi x)} \right| \; dx \\ 
	&\hspace{.25in}+\frac{1}{M^3}  \left(  \int_{\beta+1/M}^{1/2} + \int_{-1/2}^{\beta-1/M}  \right) |x-\beta|^{-3} \;  \left|  \sin{\left((n+1) 2 \pi x\right)} \;\sin{(2 \pi x)} \right| \; dx \\
	&\ll \int^{\beta+1/M}_{0} x \;dx + \frac{1}{M^3} \int_{\beta+ 1/M}^{1/2} \frac{x}{(x-\beta)^{3}} ~dx\\
	&\ll \frac{1}{M^2} + \frac{1}{M^3} \int_{\beta+ 1/M}^{1/2} \left(\frac{1}{(x-\beta)^2}+\frac{\beta}{(x-\beta)^{3}} \right) ~dx\\
	&\ll \frac{1}{M^2}.
	\end{align*}
	It remains to show that the expression in \eqref{fejer} is also $\ll 1/M^2$. First note that we may write 
	\begin{equation}\label{Un}
	U_n(\cos(2\pi x)) = \dfrac{\sin((n+1)2\pi x)}{\sin(2\pi x)}.
	\end{equation} 
	Moreover, the Fejer kernel may alternatively be expressed as
	
	\[
	\Delta_{M}(x) = \frac{1}{M}\sum_{k=0}^{M-1}D_k(x),
	\]
where $D_k(x)$ is the $k$-th order Dirichlet kernel that has a closed form expression given by $$ D_k(x)= 1 \; + \; 2\sum_{j=1}^{k} \cos(2\pi jx).$$ 
Thus, we may express $\Delta_{M+1}(x)$ as 
		\begin{equation}\label{Fejer-Dirichlet sum}
		\Delta_{M+1}(x) = \frac{1}{M+1} \sum_{k=0}^{M} \left(1 + 2\sum_{j=1}^{k} \cos(2\pi j x)\right).
		\end{equation} 
From \eqref{Un}, \eqref{Fejer-Dirichlet sum}, and a trigonometric sum-difference formula, we have
		\begin{align}
		\nonumber(M&+1)\int_{-1/2}^{1/2} \Delta_{M+1}(x-\beta) \; U_n(\cos 2 \pi x) \; \sin^2{(2 \pi x) } \; dx\\
		\nonumber&=(M+1)\int\limits_{-1/2}^{1/2} \Delta_{M+1}(x-\beta) \, \sin((n+1)2\pi x)\sin(2\pi x) \;dx\\
		\nonumber&=\frac{1}{2} \int\limits_{-1/2}^{1/2} \left( (M+1) + 2\sum_{k=1}^{M}\sum_{j=1}^{k} \cos(2\pi j (x-\beta)) \right) (\cos(2\pi n x)\; -\; \cos(2\pi(n+2)x)) \;dx\\
		\nonumber&=\sum_{k=1}^{M}\sum_{j=1}^{k} \int_{-1/2}^{1/2} \cos(2\pi j (x-\beta))\big(\cos(2\pi n x) -\cos(2\pi (n+2) x)\big) ~dx\\ \label{two-parts-1}
		&=\sum_{k=1}^{M}\sum_{j=1}^{k} \int_{-1/2}^{1/2} \cos(2\pi j (x-\beta))\cos(2\pi n x)\, dx\\ \label{two-parts-2}
		&\hspace{.75in} -\sum_{k=1}^{M}\sum_{j=1}^{k} \int_{-1/2}^{1/2} \cos(2\pi j (x-\beta))\cos(2\pi (n+2) x)\big) ~dx .
		\end{align}
		Using trigonometric identities, we can rewrite the sum in \eqref{two-parts-1} as
		\begin{align}\label{first-two-parts} \sum_{k=1}^{M}\sum_{j=1}^{k} \int_{-1/2}^{1/2} 	\left(  \cos(2\pi j\beta)\cos(2\pi j x) +\sin(2\pi j\beta)\sin(2\pi j x)\right) \cos(2\pi n x) \; dx.
		\end{align}
		Recall that
		\begin{equation*}
		\int_{-1/2}^{1/2}\cos(m\pi x)\cos(n \pi x) \;dx =
		\begin{cases} \frac{1}{2} & \text{if $m=n$}\\
		0 & \text{otherwise,}
		\end{cases}
		\end{equation*} 
		and that for any $m,n \in\Z$,
		\begin{equation*}
		\int_{-1/2}^{1/2}\cos(m\pi x)\sin(n \pi x) \; dx= 0.
		\end{equation*}
		Therefore, the only term that survives in the inner sum over $j$ in  \eqref{first-two-parts} is the term $j=n$. This gives \begin{equation*}
		\sum_{k=1}^{M}\sum_{j=1}^{k} \int_{-1/2}^{1/2} \cos(2\pi j (x-\beta))\cos(2\pi n x)\, dx = \frac{M-n+1}{2} \cos(2\pi n\beta),
		\end{equation*}
		and a similar calculation gives that the sum in \eqref{two-parts-2} is 
		\begin{equation*}
		\sum_{k=1}^{M}\sum_{j=1}^{k} \int_{-1/2}^{1/2} \cos(2\pi j (x-\beta))\cos(2\pi (n+2) x)\, dx =
		\dfrac{M-(n+2)+1}{2} \cos(2\pi (n+2)\beta)
		\end{equation*} if $1\leq n\leq M-2$ and equal to $0$ if $n= M-1, M$. Thus, for $1\leq n\leq M-2$,
		\begin{align}\label{Fejer-integral}
		\nonumber (M+1)&\int_{-1/2}^{1/2} \Delta_{M+1}(x-\beta) \; U_n(\cos 2 \pi x) \; \sin^2{(2 \pi x) } \; dx\\ \nonumber
		&=\frac{M-n+1}{2} \cos(2\pi n\beta) - \frac{M-(n+2)+1}{2} \cos(2\pi (n+2)\beta)\\
		&= (M-n+1) \sin(2\pi(n+1)\beta)\sin(2\pi\beta) + \cos(2\pi(n+2)\beta).
		\end{align} while for $n= M-1, M$ we have 
		\begin{align}\label{Fejer-integral2}
		\nonumber (M+1)&\int_{-1/2}^{1/2} \Delta_{M+1}(x-\beta) \; U_n(\cos 2 \pi x) \; \sin^2{(2 \pi x) } \; dx\\ 
		&=\frac{M-n+1}{2} \cos(2\pi n\beta).
		\end{align} Dividing by $2(M+1)^2$, the integral in \eqref{fejer} is $\ll 1/M^2$ for $\beta= 1/(2\pi M)$. 
Setting $\beta=0$ in  \eqref{Fejer-integral} and \eqref{Fejer-integral2}, we find that for non-negative integers $n,M$ with $M\geq 1$,

		\begin{equation*} 
		\int\limits_{-1/2}^{1/2} \left(\frac{\sin((M+1)\pi x)}{\sin(\pi x)}\right)^2 \, U_n(\cos 2\pi x)\sin^2(2\pi x) dx =
		\begin{cases}
		1 & \text{if $0\leq n<M$}\\
		\frac{1}{2} & \text{if $n=M$.}
		\end{cases}
		\end{equation*}
	We have thus shown that for $\beta=1/(2 \pi M)$, 
	$$\int_{-1/2}^{1/2} \left( B_M(x-\beta) - s(x-\beta) \right) \; U_n(\cos{2 \pi x}) \; \sin^2{(2 \pi x) } \; dx \ll \frac{1}{M^2}, $$ 
	and using $\beta=0$ in the above formulas, we have
	$$\int_{-1/2}^{1/2} \left( B_M(x) - s(x) \right) \; U_n(\cos{2 \pi x}) \; \sin^2{(2 \pi x) } \; dx \ll \frac{1}{M^2}.$$ 
	This completes the proof of the Proposition.
	
\end{proof} 

\noindent{\it Remark:} The results of Proposition \ref{improved-FC} also hold for the coefficients $\hat{F}^-_{I,M}(n)$, following appropriate minor changes, but this is not needed for our application.

\section{Proof of Theorem \ref{main}}\label{section3}
We adapt the arguments of \cite{RT2017} to prove  Theorem \ref{main}, using the stronger bound on the size of the Fourier coefficients computed in Proposition \ref{improved-FC}. To estimate the prime counting function
	$$\# \{ x\leq p< 2 x : a_p(E) = [2\sqrt{p}] \} , $$
	 we first perform the change of variable $a_p(E) = 2\sqrt{p} \cos \theta_p(E)$. Let $I_\varepsilon$ be an interval of the form $[0, \varepsilon] \subseteq [0,\pi/2]$ and $I'_\varepsilon = \left[ \cos(\varepsilon), 1 \right]$ is such that
	\begin{eqnarray*} \cos\theta_p(E) \in I'_{\varepsilon} \iff \theta_p(E) \in I_{\varepsilon}.\end{eqnarray*}  If $\varepsilon= \varepsilon(x)$  is such that 
	\begin{equation} \label{bound-for-epsilon} \cos\varepsilon \leq 1- x^{-1/2}, \end{equation}
	 then using $ x \leq p < 2x$, we have
	\begin{equation*}
     \cos \varepsilon \leq 1-\frac{1}{x^{1/2}}< 1-\frac{1}{2\sqrt{p}} < 1-\frac{\{2\sqrt{p}\}}{2\sqrt{p}} <1.
	\end{equation*} 
	Using this, we obtain the upper bound
	\begin{equation} \nonumber
	\begin{split}
	\# \{ x\leq p< 2 x : a_p(E) = [2\sqrt{p}] \} &= \# \left\{ x\leq p< 2 x : \frac{ a_p(E)}{2\sqrt{p}} = 1- \frac{\{2\sqrt{p}\}}{2\sqrt{p}} \right\}\\
	&\leq \# \{ x\leq p <2x: \cos\theta_p(E) \in I'_\varepsilon\}\\
	&= \# \{ x\leq p <2x: \theta_p(E) \in I_\varepsilon\}\\
	&= \sum\limits_{x\leq p<2x} \chi_{I_\varepsilon} (\theta_p(E)),
	\end{split}
	\end{equation} 
 where for any interval $I$, $\chi_I$ is the characteristic function of the interval.
 
 Let $\varepsilon= 1/M$ so that $I_\varepsilon= [0, 1/M]$, where $M$ is chosen later. Using the first property in Lemma \ref{lemma-RT}, we have 
\begin{equation}\label{upper-bound for char-function}
	\begin{split}
		\sum\limits_{x\leq p<2x} \chi_{I_\varepsilon} (\theta_p) &\leq  \sum\limits_{n=0}^M \hat{F}^{+}_{I_\varepsilon,M}(n)\sum\limits_{x\leq p<2x}U_n(\cos\theta_p(E))\\
		&\leq \sum\limits_{n=0}^M |\hat{F}^{+}_{I_\varepsilon,M}(n)| \left|\sum\limits_{x\leq p<2x}U_n(\cos\theta_p(E))\right|.
	\end{split}
\end{equation}
To estimate the quantity $\sum\limits_{x\leq p<2x}U_n(\cos\theta_p(E))$, as in  \cite{RT2017}, we get sharper estimates by weighting the contribution from primes using a test function which is a pointwise upper bound for the characteristic function on $[x,2x]$. Let \begin{equation} \label{def-g}
	g(y) =\begin{cases}
		\exp\left( \frac{4}{3} + \frac{1}{(y-\frac{1}{2})(y-\frac{5}{2})}\right) &\text{if } \frac{1}{2} < y < \frac{5}{2},\\
		0 &\text{ otherwise}
	\end{cases}
\end{equation} and $g_x(y) = g(y/x)$. Using the bound $1\leq \frac{\log p}{\log x}$ for all $x\leq p <2x$, this allows us to write 
\begin{equation}\label{RS-1}
	\left|\sum\limits_{x\leq p<2x}U_n(\cos\theta_p(E))\right| \leq \frac{1}{\log x} \left| \sum\limits_{p}U_n(\cos\theta_p(E))g_x(p)\log p\right|.
\end{equation}
We next use a result of \cite{RT2017} stated in the following form.
\begin{prop}[{\cite[Proposition 3.5]{RT2017}}]\label{newRT} For each $n \geq 0$, 
	assume that the L-function $L(s, \text{Sym}^n(E))$ is entire (with the exception of a simple pole at $s=1$ when $n=0$), satisfies the functional equation \eqref{FE-Sym} and the Generalized Riemann Hypothesis. Then, we have
	\begin{equation}\label{RS-2}
		 \sum\limits_{p}U_n(\cos\theta_p(E))g_x(p)\log p \; \ll_E \; \delta_{n,0} x + \sqrt{x} n\log n 
	\end{equation}
	where $\delta_{n,0}=1$ if $n=0$ and 0 otherwise.
\end{prop}

\begin{remark*}
Proposition 3.5 of \cite{RT2017} has the additional hypothesis that $N_E$ is square-free. We describe how to remove this hypothesis in Appendix \ref{appendix}. \end{remark*}
	
	Using \eqref{RS-1} and \eqref{RS-2} in \eqref{upper-bound for char-function} we now have 
	\begin{equation*}
	\sum\limits_{x\leq p<2x} \chi_{I_\varepsilon} (\theta_p(E)) \ll \frac{1}{\log x}\sum\limits_{n=0}^M |\hat{F}^{+}_{I_\varepsilon,M}(n)|\left(\delta_{n,0} x + \sqrt{x} n\log n \right).
	\end{equation*} 
	We now use Proposition \ref{new-Fourier} to bound the Fourier coefficients $ \hat{F}^{+}_{I_\varepsilon,M}(n)$. Doing so, the right hand side of the above equation is \begin{equation*}
	\begin{split}
	\ll & \dfrac{1}{M^2 \log x}\left( x + \sqrt{x} \sum_{n=1}^{M}n\log n\right)\\
	\ll &\frac{x}{M^2\log x} + \frac{x^{1/2}\log M}{\log x}.
	\end{split}
	\end{equation*} Letting $M=\left \lceil \dfrac{x^{1/4}}{(\log x)^{1/2}} \right\rceil $, we see that \eqref{bound-for-epsilon} is satisfied, and we have 
	\begin{equation*}
	\# \left\{ x \leq p < 2x \;:\; a_p(E) = [2 \sqrt{p}] \right\} \leq
	\sum\limits_{x\leq p<2x} \chi_{I_\varepsilon} (\theta_p(E)) \ll x^{1/2},
	\end{equation*} which completes the proof of Theorem \ref{main}.

\appendix
\section{}\label{appendix}
\subsection{Proof of Proposition \ref{newRT}}\label{appendixproof}
We now extend the proof of Proposition 3.5 of \cite{RT2017} to all non-CM elliptic curves over $\Q$, without assuming that $N_E$ is square-free, i.e. that the bad primes are primes of multiplicative reduction. This necessitates a bound on the conductor of $\text{Sym}^n(E)$, and computing the local factors at the bad primes of all reduction types. We summarize in sections \ref{A2} and \ref{A3} the work of Martin and Watkins \cite{MW} which gives us these estimates.

	Let $N_{E, n}$ be the conductor of $\text{Sym}^n(E)$. Notice that in this notation, $N_{E,1} = N_E$  (and $N_{E,0}=1$). It is conjectured that the $L(s, \text{Sym}^n(E))$ satisfy the functional equation
\begin{equation}\label{FE-Sym}
\Lambda(s, \text{Sym}^n(E)) = \varepsilon_n \, \Lambda(1-s, \text{Sym}^n(E)) ,
\end{equation}
where the root number $\varepsilon_n \in \C$ has absolute value 1, 
and the completed L-function is
$$
\Lambda(s, \text{Sym}^n(E)) = N_{E, n}^{s/2} \gamma(s, \text{Sym}^n(E)) L(s, \text{Sym}^n(E)),
$$
with the gamma factor 
$$
\gamma(s, \text{Sym}^n(E)) = \begin{cases} \left( 2^{1-s}  \pi^{-s} \right)^{(n+1)/2} \displaystyle \prod_{j=1}^{(n+1)/2} \Gamma\left( s + (j-1/2)(n-1) \right)  &
\mbox{if $n$ is odd} \\    & \\ \displaystyle  \pi^{-(s+n_2)/2} \Gamma((s+n_2)/2) \left( 2^{1-s}  \pi^{-s} \right)^{n/2} \prod_{j=1}^{n/2} \Gamma\left( s + j(n-1) \right) & \mbox{if $n$ is even.} 
\end{cases}$$
In the above, $n_2 = n/2 \mod 2.$

	 As in \cite{RT2017}, define the numbers $\Lambda_{\text{Sym}^n (E})(j)$ by 
	\begin{equation*}
	-\frac{L'}{L}(s, \text{Sym}^n(E)) =\sum_{j=1}^{\infty} \frac{\Lambda_{\text{Sym}^n (E)}(j)}{j^s}, \qquad \mbox{Re}(s) >1.
	\end{equation*} 
	For primes $p$ not dividing $N_E$, and $m \geq 1$, it is a straightforward computation to show that 
	\begin{equation} \label{goodprimes} \Lambda_{\text{Sym}^n(E)}(p^m) = U_n(\cos (m\theta_p(E)))\log p. \end{equation}
	Also, $\Lambda_{\text{Sym}^n(E)}(j)$ is zero when $j$ is not power of a prime. Thus, 
	\begin{eqnarray*}
	\sum\limits_{p}U_n(\cos\theta_p(E))g_x(p)\log p &=& \sum_{j=1}^{\infty} \Lambda_{\text{Sym}^n(E)}(j)g_x(j) - \sum\limits_{\substack{ j\geq 1 \\j=p^m, \,  m\geq 2\\ \text{or } j=p, \, p|N_E}} \Lambda_{\text{Sym}^n(E)}(j)g_x(j) \\&&+ \sum\limits_{p|N_E} U_n(\cos\theta_p(E))g_x(p)\log p.
	\end{eqnarray*}
We now show that for any integer $j$, we have
	$$\Lambda_{\text{Sym}^n(E)}(j) \ll (n+1) \Lambda(j),$$  where $\Lambda(j)$ is the usual von Mangoldt function.
	If $j=p^m$ and $p \nmid N_E$, the result is clear by \eqref{goodprimes}. Suppose now that $j=p^m$, for $p \mid N_E$ and $m \geq 1$.
	Using the formulas \eqref{mult}, \eqref{add1} and \eqref{add2} which give the Euler products at the bad primes of $-\frac{L'}{L}(s, \text{Sym}^n(E))$ for multiplicative, potentially multiplicative and potentially good reduction, the result follows easily.
Thus, we have 
		\begin{align}
	&  \sum\limits_{\substack{ j\geq 1 \\j=p^m, \,  m\geq 2\\ \text{or } j=p, \, p|N_E}} \Lambda_{\text{Sym}^n(E)}(j)g_x(j) - \sum\limits_{p\nmid N_E} U_n(\cos\theta_p(E))g_x(p)\log p  \nonumber \\
    &\hspace{1in} \ll (n+1) \left(\sum\limits_{\substack{x/2<p^m<5x/2\\ m\geq 2\\ p\nmid N_E}} \log p + \sum\limits_{\substack{x/2\leq p^m<5x/2\\ p|N_E}} \log p + \sum\limits_{p|N_E} \log p \nonumber \right)
    \\&\hspace{1in}
    \ll (n+1)\left( \sum\limits_{\substack{x/2<p^m <5x/2 \\ m\geq 2}}\log p + \log N_E \right)
	 \ll_E n\sqrt{x} \label{error-prime-powers}.
	\end{align}

Then, as in \cite{RT2017}, we have to estimate
\begin{eqnarray*}
 \sum_{j=1}^\infty  {\Lambda_{\text{Sym}^n (E)}(j)} \; g_x(j).
\end{eqnarray*}
This is done by first writing an explicit formula for the non-smoothed sum
$$\psi_{\text{Sym}^n(E)}(x) = \sum_{j \leq x}  {\Lambda_{\text{Sym}^n (E)}(j)},$$
and evaluating the residues at the poles, coming from the zeroes of $L(s, {\text{Sym}}^n(E))$ in the critical strips. In all those estimates, the authors use the fact that $N_{E,n}=N_E^n$, which leads to the bound $\log(N_{E,n}) \leq n \log{N_E}$. From  \eqref{bound-C} of section \ref{A2},  it follows  that without any hypothesis on the reduction type of $E$ at the bad primes, we have
$$\log{N_{E,n}} \ll n \log{N_E},$$
where the implied constant is absolute.
The argument of  \cite[Section 8]{RT2017} becomes
\begin{eqnarray}\nonumber
\sum_{j=1}^{\infty} \Lambda_{\text{Sym}^n(E)}(j)g_x(j) &=& \delta_{n,0} x \int_{0}^\infty g(t) \, dt + O\left( \sqrt{x} \,( n \log n + \log N_{E,n}) \right) \\
\label{Prop3.5estimate}
&\ll_E&  \delta_{n,0} x + \sqrt{x} \,n \log(n),\end{eqnarray}
where we recall that $g$ is defined by \eqref{def-g}.

	Comparing \eqref{Prop3.5estimate} and \eqref{error-prime-powers}, we complete the proof of Proposition \ref{newRT} without assumption on $N_E$.

	\subsection{The conductors $N_{E,n}$}\label{A2}
	
	We now summarize the results concerning the conductors $N_{E,n}$ from \cite[Section 3]{MW}. 
For each prime $p \mid N$, fix $\ell \neq p$, and let  $T_\ell(E)$ denote the Tate module at $\ell$. Let
$H_\ell(E) = \mbox{Hom}(T_\ell(E) \otimes \Q_\ell, \Q_\ell)$ and $I_p \leq \mbox{Gal}(\overline{\Q_p}/\Q_p)$ be the local inertia group at $p$. Define
$\epsilon_n(I_p)$ to be  the co-dimension of the subspace of $\mbox{Sym}^n(H_\ell(E))$ fixed by $I_p$. Then, we have that
\[
N_{E,n} = \prod_{p \mid N_E} p^{\epsilon_n(I_p) + \delta_n(p)},
\]
where  $\delta_n(p)$ is the wild part of the conductor.  If $p$ is a prime of multiplicative reduction, then $\epsilon_n(I_p) = n$ and $\delta_n(p) =0$ for all $n$, and then $N_{E,n} = N_E^n$ when $N_E$ is square-free as assumed in \cite[Conjecture 1.1 (a)]{RT2017}.

For the other cases, we first remark that the wild conductor $\delta_n(p)=0$ when $p \geq 5$ for all reduction types. For $p$ a prime of  potentially multiplicative reduction, we have that $\epsilon_n(I_p)=n+1$ if $n$ is odd, and $\epsilon_n(I_p)=n$ if $n$ is even. The wild conductors
$\delta_n(p)$ are always 0, except for the case $p=2$ and $n$ odd, where we have $\delta_n(2) = \frac{n+1}{2} \delta_1(2).$

For $p$ a prime of  potentially good reduction, the value of $\epsilon_{n}(p)$ depends on the inertia group of the local extension $G_p =\mbox{Gal}(\Q_p(E_\ell)/\Q_p)$, and the congruence of $n$ modulo 12. The values of $\epsilon_{n}(p)$ in all cases that arise are given in Table 1 of \cite{MW}, and we always have $0 \leq \epsilon_n(p) \leq n+1$. The wild conductors $\delta_n(2)$ and $\delta_n(3)$ are given in Tables 2 and 3 of \cite{MW}, and we have that $\delta_n(2) \leq 2(n+1)$ and $\delta_n(3) \leq (n+1)/2.$

We then have the bound 
\begin{eqnarray} \label{bound-C} N_{E,n} \leq   2^{6(n+1)}  3^{(n+1)/2} \; \prod_{\substack{p \mid N_E}} p^{(n+1)} \ll_E N_E^{6n}. \end{eqnarray} 
 We also remark that the computation of the conductor $N_{E,n}$ in \cite{MW} is the idea presented in \cite[Section 5]{Rouse} applied to the special case of elliptic curves.
 
 \subsection{The local factors at the bad primes}\label{A3}

Next, we summarize the results concerning Euler factors at primes of bad reduction from \cite[Section 3]{MW}. Since we are using the normalized L-function $L(s, E)$ defined by \eqref{normL}, and 
Martin and Watkins are using the non-normalized L-function, we adjust their result accordingly using the fact that
$$
L(s, \mbox{Sym}^n(E)) = L_{\text{non-norm}}(s + n/2, \mbox{Sym}^n(E)).
$$
Let $L_p(s, \Sym^n(E))$ be the Euler factors at the bad primes $p$  of $L(s, \mbox{Sym}^n(E))$.

If $p$ is a prime of multiplicative reduction or potentially multiplicative reduction, then $$ L_p(s, \Sym^n(E))= \left( 1 - \frac{a_{p,n}}{p^{s+n/2}} \right)^{-1},$$ 
where $a_{p,n} \in \left\{ 0, \pm 1 \right\}$, and then
\begin{equation} \label{mult}
\sum_{m=1}^{\infty} \frac{\Lambda_{\Sym^n(E)}(p^m)}{p^{ms}} = \sum_{m=1}^{\infty} \frac{ \log p}{p^{sm}} \frac{a_{p,n}^m}{p^{nm/2}}.\end{equation}

In the case where $p$ is a prime with potentially good reduction, there are 2 cases depending if the local decomposition group 
$G_p$ is abelian or not. When $G_p$ is abelian then the local inertia group is cyclic of order $d$, where $d=2,3,4$ or $6$, and 
$$
L_p(s, \Sym^n(E)) = \prod_{{0\leq k\leq n}\atop {d|(2k-n)}} \left( 1- \frac{\beta_p(E)^{n-k} \overline{\beta}_p(E)^k}{p^{s+n/2}} \right)^{-1},
$$
where $\beta_p(E)$ is obtained by counting points on a $p^r$-th quadratic twist of $E$ (which is non-singular) where $r$ depends on the $p$-valuation of the coefficients of $E$. It follows that $|\beta_p(E)|=p^{1/2}$. Then,
\begin{equation} \label{add1}
\sum_{m=1}^{\infty} \frac{\Lambda_{\Sym^n(E)}(p^m)}{p^{ms}} = \sum_{m=1}^{\infty} \frac{ \log p}{p^{sm}} \frac{1}{p^{nm/2}}
\sum_{\substack {0\leq k\leq n \\ {d|(2k-n)}}} \left( \beta_p(E)^{n-k} \overline{\beta}_p(E)^k \right)^m.
\end{equation}
On the other hand, when $G_p$ is non-abelian,  we have
$$
L_p(s, \Sym^n(E)) = \left(  1 - \frac{(\pm 1) (-p)^{n/2}}{p^{s+n/2}} \right)^{-(n+1-\epsilon_n(I_p))}
$$
and
\begin{equation} \label{add2}
\sum_{m=1}^{\infty} \frac{\Lambda_{\Sym^n(E)}(p^m)}{p^{ms}}  = \sum_{m=1}^\infty \frac{(\pm 1)^m (-1)^{mn/2} \log p}{p^{sm}}
 \left( {n+1-\epsilon_n(I_p)} \right).
\end{equation}

\bibliographystyle{alpha}
\bibliography{ExtremalPrimesBIB.bib}

\newcommand{\etalchar}[1]{$^{#1}$}
\begin{thebibliography}{AHJ{\etalchar{+}}18}

\bibitem[AHJ{\etalchar{+}}18]{AHJKL}
Anthony Agwu, Phillip Harris, Kevin James, Siddarth Kannan, and Huixi Li.
\newblock Frobenius distributions in short intervals for {CM} elliptic curves.
\newblock {\em J. Number Theory}, 188:263--280, 2018.

\bibitem[BK16]{BC}
Alina Bucur and Kiran~S. Kedlaya.
\newblock An application of the effective {S}ato-{T}ate conjecture.
\newblock In {\em Frobenius distributions: {L}ang-{T}rotter and {S}ato-{T}ate
  conjectures}, volume 663 of {\em Contemp. Math.}, pages 45--56. Amer. Math.
  Soc., Providence, RI, 2016.

\bibitem[CHT08]{CHT2008}
Laurent Clozel, Michael Harris, and Richard Taylor.
\newblock Automorphy for some {$l$}-adic lifts of automorphic mod {$l$}
  {G}alois representations.
\newblock {\em Publ. Math. Inst. Hautes \'Etudes Sci.}, (108):1--181, 2008.
\newblock With Appendix A, summarizing unpublished work of Russ Mann, and
  Appendix B by Marie-France Vign\'eras.

\bibitem[Deu41]{Deuring}
Max Deuring.
\newblock Die {T}ypen der {M}ultiplikatorenringe elliptischer
  {F}unktionenk\"orper.
\newblock {\em Abh. Math. Sem. Hansischen Univ.}, 14:197--272, 1941.

\bibitem[Gib17]{Giberson}
Luke~M. Giberson.
\newblock {\em Average {F}robenius {D}istributions for {E}lliptic {C}urves:
  {E}xtremal {P}rimes and {K}oblitz's {C}onjecture}.
\newblock ProQuest LLC, Ann Arbor, MI, 2017.
\newblock Thesis (Ph.D.)--Clemson University.

\bibitem[GJ18]{GJ}
Luke Giberson and Kevin James.
\newblock An average asymptotic for the number of extremal primes of elliptic
  curves.
\newblock {\em Acta Arith.}, 183(2):145--165, 2018.

\bibitem[HSBT10]{HSBT2010}
Michael Harris, Nick Shepherd-Barron, and Richard Taylor.
\newblock A family of {C}alabi-{Y}au varieties and potential automorphy.
\newblock {\em Ann. of Math. (2)}, 171(2):779--813, 2010.

\bibitem[JP17]{JP2017}
Kevin James and Paul Pollack.
\newblock Extremal primes for elliptic curves with complex multiplication.
\newblock {\em J. Number Theory}, 172:383--391, 2017.

\bibitem[JTT{\etalchar{+}}16]{James2016}
Kevin James, Brandon Tran, Minh-Tam Trinh, Phil Wertheimer, and Dania Zantout.
\newblock Extremal primes for elliptic curves.
\newblock {\em J. Number Theory}, 164:282--298, 2016.

\bibitem[LT76]{LT1976}
Serge Lang and Hale Trotter.
\newblock {\em Frobenius distributions in {${\rm GL}_{2}$}-extensions}.
\newblock Lecture Notes in Mathematics, Vol. 504. Springer-Verlag, Berlin-New
  York, 1976.
\newblock Distribution of Frobenius automorphisms in
  ${{\rm{G}}L}_{2}$-extensions of the rational numbers.

\bibitem[Mon94]{Mont}
Hugh~L. Montgomery.
\newblock {\em Ten lectures on the interface between analytic number theory and
  harmonic analysis}, volume~84 of {\em CBMS Regional Conference Series in
  Mathematics}.
\newblock Published for the Conference Board of the Mathematical Sciences,
  Washington, DC; by the American Mathematical Society, Providence, RI, 1994.

\bibitem[Mur85]{VKMurty}
V.~Kumar Murty.
\newblock Explicit formulae and the {L}ang-{T}rotter conjecture.
\newblock {\em Rocky Mountain J. Math.}, 15(2):535--551, 1985.
\newblock Number theory (Winnipeg, Man., 1983).

\bibitem[MW06]{MW}
Phil Martin and Mark Watkins.
\newblock Symmetric powers of elliptic curve {$L$}-functions.
\newblock In {\em Algorithmic number theory}, volume 4076 of {\em Lecture Notes
  in Comput. Sci.}, pages 377--392. Springer, Berlin, 2006.

\bibitem[Rou07]{Rouse}
Jeremy Rouse.
\newblock Atkin-{S}erre type conjectures for automorphic representations on
  {${\rm GL}(2)$}.
\newblock {\em Math. Res. Lett.}, 14(2):189--204, 2007.

\bibitem[RT17]{RT2017}
Jeremy Rouse and Jesse Thorner.
\newblock The explicit {S}ato-{T}ate conjecture and densities pertaining to
  {L}ehmer-type questions.
\newblock {\em Trans. Amer. Math. Soc.}, 369(5):3575--3604, 2017.

\bibitem[Tay08]{Taylor2008}
Richard Taylor.
\newblock Automorphy for some {$l$}-adic lifts of automorphic mod {$l$}
  {G}alois representations. {II}.
\newblock {\em Publ. Math. Inst. Hautes \'Etudes Sci.}, (108):183--239, 2008.

\end{thebibliography}

\end{document}